\documentclass{article}
\usepackage{amsmath,amssymb,amsthm, algorithm, algorithmic, enumerate,listings,marvosym,hyperref,eucal,verbatim,amscd}

\title{Convergence of the Zagier type series for the Cauchy kernel}
\author{Nina Sakharova\thanks{The article was prepared within the framework of a subsidy granted to the HSE by the Government of the Russian Federation for the implementation of the Global Competitiveness Program.  }}
\date{}

\newtheorem{theorem}{Theorem}
\newtheorem{lemma}{Lemma}
\newtheorem{corollary}{Corollary}

\theoremstyle{definition}

\theoremstyle{remark}
\newtheorem{remark}{Remark}

\usepackage[pdftex]{graphicx}
\usepackage{bm}
\usepackage{amsfonts}
\usepackage{amsmath}
\usepackage{amssymb}
\usepackage{amsthm}

\renewcommand{\pmod}[1]{(\textmd{mod}\hspace{.5mm}#1)}

\begin{document}

\maketitle

\begin{center}
\large{ \textit{Department of mathematics, National Research University Higher School of
Economics, Vavilova str. 7, 117312, Moscow, Russia.\\
saharnina@gmail.com}}
\end{center}

\begin{abstract}
In 1975 prof. Don Zagier derived a preliminary formula for the trace of the Hecke operators acting on the space of cusp forms (\cite{5}, \cite{6}). Actually, it is an expression in terms of an integral over a fundamental domain of $SL_2(\mathbb{Z}).$ His theorem tells us that if $f$ is a cusp form of weight $k$, then we can identify the Peterson scalar product of $f$ and a certain series $\omega_m(z_1,\bar{z_2}, k)$ with the action of the Hecke operator $T(m)$ on the function $f$, up to a constant that depends only on $k$ and $m$. It follows that $\omega_m(z_1,\bar{z_2}, k)$ is kind of "kernel function" for the operator $T(m)$.

Don Zagier proved this theorem using the Rankin-Selberg method. Other evidence was proposed by prof. A. Levin. He suggested to construct a Cauchy kernel. Formally, the Cauchy kernel expressed by the series, which doesn't converge absolutely.
The main purpose of this paper is to extend this series to the edge of convergence by analytic continuation. The second part of the paper is devoted to getting an expression for differential form of logarithm of difference of two $j$-invariant values $|j(z_1)-j(z_2)|$.

\end{abstract}

\section{Introduction}

By definition, put
$$\mu_{\gamma}(z_1, -z_2)=cz_1z_2+dz_2-az_1-b,$$
where $z_1,z_2$ are arbitrary points  from the upper-half plane $\mathfrak{H} = \left\{z: \Im(z) \geq 0\right\}$ and $\gamma=\left(\begin{array}{ccc}
	a&b\\
	c&d\\
	\end{array}\right)$ is an integer matrix
	with determinant~$m$.\\

In \cite{5},~ \cite{6}, Don Zagier proved that the Hecke operator $T_k(m)$ on the space of cusp forms of weight $2k$ can be defined by a kernel  $\omega_m(z_1,\bar{z_2}, k)$:
\begin{multline}
\omega_m(z_1,\bar{z_2}, k)= \sum_{ad-bc=m} \frac{1}{\mu_{\gamma}(z_1, -\bar{z_2})^k}=\\=
\sum_{ad-bc=m} \frac{1}{\left(cz_1\bar{z_2}+d\bar{z_2}-az_1-b\right)^{k}}=
\sum_{ad-bc=m}\frac{1}{\left(\bar{z_2}-\frac{az_1+b}{cz_1+d}\right)^{k}(cz_1+d)^{k}},
\end{multline}
where the sum is taken over all integer matrices	
	$\left(\begin{array}{ccc}
	a&b\\
	c&d\\
	\end{array}\right)$ with determinant~$m$.\\

It is easy to prove that for an even integer $k > 2$ the series $\omega_m(z_1,\bar{z_2}, k)$ is a cusp form of weight $k$.

\begin{theorem} [D. Zagier]
Let $F$ be a fundamental domain for the modular group $\Gamma=SL_2(\mathbb{Z})$ in $\mathbb{H}$ and let  $$C_k=\frac{(-1)^{k/2} \pi}{2^{k-3}(k-1)};$$ then, for every holomorphic cusp form  $f$ of weight $k$, we have
$$\int_{F}f(z_1)\overline{\omega_m(z_1,\bar{z_2}, k)}(\Im(z_1))^{k-2} dz_1\bar{dz_1}=f\ast \omega_m(z_2)=C_k m^{1-k}(T_k(m)f)(z_2).$$   \label{T1}
\end{theorem}

\begin{remark}
Since $m^{k-1}\omega_m=T_k(m)h_1$, it is suffices to consider the case when the determinant $m=1$.
\end{remark}

D. Zagier proved Theorem 1 using the Rankin-Selberg method. Another way to prove this theorem is to construct a Cauchy kernel. Formally, it is expressed by the series
\begin{equation}
\Xi(z_1,z_2, k)=\sum_{\gamma \in \Gamma}\dfrac{1}{\mu_{\gamma}(z_1, -z_2)\mu_{\gamma}(z_1, -\bar{z_2})^{2k-1}}. \label{Xi}
\end{equation} If $k=1$, then the series $\Xi(z_1,z_2, k)$ doesn't converge absolutely, but it is just at the edge of convergence, since $\sum \limits_{\gamma \in \Gamma} \dfrac{1}{|\mu_{\gamma}(z_1,- z_2)|^{s}|\mu_{\gamma}(z_1,-\bar{z_2})|^{s}}$ converges for any $s$ such that $\Re(s)>1$.

Following the lead of E. Hecke, we investigate the series
\begin{equation}
\Xi_n(z_1,z_2, s)= \sum_{\gamma \in \Gamma} \frac{\overline{\mu_{\gamma}(z_1, -z_2)}^n~ \overline{\mu_{\gamma}(z_1, -\bar{z_2})}^n}{|\mu_{\gamma}(z_1, -z_2)|^{2s}~|\mu_{\gamma}(z_1, -\bar{z_2})|^{2s}}, \label{Xi2}
\end{equation}
	where $n$ is a non negative integer, $s$ is a complex number. The series $\Xi_n(z_1,z_2, s)$ is absolutely convergent iff $\Re(s)>\frac{n+1}{2}$ and if $s=n=k$, then $\Xi_k(z_1,z_2, k)=\Xi(z_1,z_2, k)$.\\
  	 	
The aim of this paper is to prove that the series (\ref{Xi2}) can be continued as a holomorphic function in $s$ to the point $s=n=1$. As a consequence, in Section 4, we will get an expression for $d_{z_1} \log~|j(z_1)-j(z_2)|^2$ in terms of series $\Xi(z_1,z_2, 1)$.

\section{Asymptotics of the series $ \Xi_n(z_1,z_2, s)$ }
    \begin{theorem} The series $\Xi_n(z_1,z_2, s)$ can be analytically continued to the point $s=n=1$. \label{T3}
\end{theorem}
    \begin{proof}
\begin{enumerate}[(i)]

\item We can assume without loss of generality that the determinant $m$ of the matrix
$\left(\begin{array}{ccc}
	a&b\\
	c&d\\
	\end{array}\right)$ is 1.
\item Splitting the sum (\ref{Xi2}) into subsums with respect to the various values of $c$ and combining each summand with its negation, we get
$$\Xi_n(z_1,z_2, s)=\Xi_n^{0}(z_1,z_2, s)+2~\Xi_n^{c}(z_1,z_2, s),$$
where the series $\Xi_n^{0}(z_1,z_2, s)$ corresponds  to $c=0$ and the series $\Xi_n^{c}(z_1,z_2, s)$ corresponds  to $c>0$.

%then
%\begin{equation}
%\Xi_n^{0}(z_1,z_2, s)=\sum_{\substack{ad=1, \\ b \in \mathbb{Z}}}\frac{(d\bar{z_2}-a%\bar{z_1}-b)^n}{|dz_2-az_1-b|^{2s}}\frac{(dz_2-a\bar{z_1}-b)^{n}}{|dz_2-a\bar{z_1}-b|^{2s}}, \label{Xio}
%5\end{equation}
%\begin{multline}
%\Xi_n^{c}(z_1,z_2, s)=2\sum_{c>0}\sum_{\substack{a,d,\\~ad \equiv 1 \pmod c}}\frac{c^{2n}}%{|c|^{-4s}}
%			\frac{((c\bar{z_1}+d)(c\bar{z_2}-a)+1)^n}{|(cz_1+d)(cz_2-a)+1|^{2s}}\frac{((c\bar{z_1}+d)(cz_2-a)+1)^n}{|(c\bar{z_1}+d)(cz_2-a)+1|^{2s}}. \label{Xic}
% \end{multline}\\
		
%%%%%%%%%%%%%%%%%%%%%%%%%%%%%%%%%%%%%%%%%%%%%%%%%%%%%%%%%%%%%%%%%c=0%%%%%%%%%%%%%%%%%%%%%%%
\item Case 1. If $c=0$, then either $a=d=1$ or $a=d=-1$ and summation over $b \in \mathbb{Z}$ is unrestricted. Therefore,
\begin{equation}
\Xi_n^{0}(z_1,z_2, s)=4\sum_{ b> 0}\frac{(\bar{z_2}-\bar{z_1}-b)^n~(z_2-\bar{z_1}-b)^n}{|z_2-z_1-b|^{2s}~|\bar{z_2}-z_1-b|^{2s}}. \label{Xi0}
\end{equation}

We have $\frac{(\bar{z_2}-\bar{z_1}-b)^n(z_2-\bar{z_1}-b)^n}{|z_1-z_2-b|^{2s}|\bar{z_2}-z_1-b|^{2s}} = \frac{1}{b^{4s-2n}}+O\left(\frac{1}{b^{4s-2n+1}} \right)$ as $b\rightarrow \infty$, whence the sum (\ref{Xi0})	is absolutely convergent in the half-plane $\Re(s)	
>\frac{2n+1}{4}$ and has a simple pole with residue 1 at $s=\frac{2n+1}{4}$. \\

%%%%%%%%%%%%%%%%%%%%%%%%%%%%%%%%%%%%%%%%%%%%%%%%%%%%%%%%%%%%%%%%%c>0%%%%%%%%%%%%%%%%%%%%%%%
 \item  Case 2, $c > 0$.

Firstly, note that  $$\mu_{\gamma}(z_1,z_2)=c^{-1}\left[(cz_1+d)(cz_2-a)+(ad-bc)\right]=c^{-1}\left[(cz_1+d)(cz_2-a)\right]+1/c.$$

It is easy to check that if the series
\begin{equation}
\tilde{\Xi}_n^{c}(z_1,z_2, s)=\sum_{c> 0}\sum_{\substack{a,d,\\ad \equiv 1 \pmod c}}\frac{(\overline{\mu_{\gamma}(z_1,-z_2)-1/c})^n~ \overline{\mu_{\gamma}(z_1,-\bar{z_2})-1/c})^n}{|\mu_{\gamma}(z_1,-z_2)-1/c|^{2s}~ |\mu_{\gamma}(z_1,-\bar{z_2})-1/c|^{2s}}
\end{equation}
can be analytically continued to some point, then the sum	$\Xi_n^{c}(z_1,z_2, s)$ can be analytically continued to this point as well.

Therefore, we will consider the sum
\begin{equation}
\tilde{\Xi}_n^{c}(z_1,z_2, s)=\sum_{c> 0}\sum_{\substack{a,d,\\ad \equiv 1 \pmod c}}\frac{c^{-2n}}{|c|^{-4s}}\frac{((c\bar{z_1}+d)(c\bar{z_2}-a))^n}{|(cz_1+d)(cz_2-a)|^{2s}}\frac{((c\bar{z_1}+d)(cz_2-a))^n}{|(cz_1+d)(c\bar{z_2}-a)|^{2s}}. \label{Xit}
\end{equation}

\item Consider the classical series \cite{7}
$$S_n(z,y,s)=\sum_{\nu}^{*}\frac{(\bar{z}+\nu)^n}{|z+\nu|^{2s}}e^{-\nu y},$$
here $n$ is an integer, $s$ is a complex number, $y$ is a real number, and  $\sum \limits^{*}$ denotes the sum taken over all integers $\nu \neq \tau$. The series $S_n(z,~ y,~ s)$ is absolutely convergent iff $\Re (s) > \frac {n+1}{2}$.
\begin{lemma} \cite{7}
Suppose that $\Re (s) > \frac {n+1}{2}$ and $y=0$; then the Fourier series expansion of the function $S_n(z, y, s)$ is given by
\begin{multline}
 S_n\left(z, 0, s\right) = \frac{2^{1+n-2s} \pi \Gamma(2s-n-1)}{(i)^n \Gamma(s) \Gamma(s-n)}\Im (z)^{1+n-2s}	+
  \\  +\frac{(i)^n \sqrt{\pi}}{(2)^{n-1} \Gamma(s)}\sum_{r\neq 0} \left|\pi  r\right|^{2s-n-1}\Phi_{\mathrm{sgn (r)}}\left(\left|\pi r \Im (z)\right|,~ n,~ s-n-\frac{1}{2}\right)e^{2 \pi i r \Re(z)} \label{S_n},
 	 	 	\end{multline}
where
\begin{equation}
\Phi_{sgn (r)}\left(Y, n, \lambda\right) = e^{2 sgn (r) Y}\frac{d^n}{dY^n}\left[e^{-2 sgn (r) Y}Y^{-\lambda}K_\lambda(2Y)\right],~~~Y>0,
\end{equation}
and $$K_\lambda(2Y)=\frac{1}{2}\int_0^{\infty}e^{-Y\left(t+\frac{1}{t}\right)}t^{\lambda-1}dt $$
is the modified Bessel function of the second kind (the MacDonald function).
\end{lemma}	
 			
\item Assume that $a=-a_0+ck$, $d=d_0+lc$, ~$k, l \in \mathbb{Z}$. If $a_0d_0 \equiv 1~ \pmod c$, then $(a_0, c)=1$. Hence the right-hand side of (\ref{Xit}) equals
 \begin{multline}
\sum_{\substack{c\in\mathbb{Z}\\ c > 0}}\sum_{\substack{a,d,\\ad \equiv 1 \pmod c}}c^{4s-2n}\frac{1}{|cz_2-a|^{4s-2n}}\frac{(c\bar{z_1}+d)^{2n}}{|cz_1+d|^{4s}}=\\
=\sum_{c>0}\sum_{\substack{1< a_0 < c\\(a_0, c)=1 \\a_0d_0 \equiv 1 \pmod c}}\frac{1}{c^{4s-2n}}S_0\left(z_2+\frac{a_0}{c},~ 0,~ 2s-n \right)S_{2n}\left(z_1+\frac{d_0}{c},~ 0,~ 2s \right) \label{SS}.
\end{multline}\\

\item The sum ~$\tilde{\Xi}_n^{c}(z_1,z_2, s)$~ satisfies the periodicity property $\tilde{\Xi}_n^{c}(z_1+\nu,z_2+\nu', s)=\tilde{\Xi_n^{c}}(z_1,z_2, s)$ for ~$\nu,\nu' \in \mathbb{Z}$ and hence $\tilde{\Xi}_n^{c}(z_1,z_2, s)$ has a Fourier expansion of the form
\begin{multline}
\tilde{\Xi}_n^{c}(z_1,z_2, s)=\sum_{c>0}A^0(c)+\sum_{c>0}\sum_{\substack{r \neq 0\\ r'=0}}A^r( c)e^{2\pi i r\Re(z_2)}+\\
+\sum_{c>0}\sum_{\substack{r=0 \\ r' \neq 0}}A^{r'}(c)e^{2\pi i r'\Re(z_1)} +\sum_{c>0}\sum_{\substack{r \neq 0\\ r'\neq 0}}A^{r, r'}(c)e^{2\pi i (r\Re(z_1)+r'\Re(z_2))}. \label{Fourier}
\end{multline}
The coefficients in the right-hand side of (\ref{Fourier}) are given by Lemma 1: \\	 	 	

\begin{enumerate}
 \item for $r=r'=0$, the constant term $\sum \limits_{c>0}A^0(c)$ of the Fourier  expansions  is
$$
\sum_{c>0}A^0(c)= \sum_{c>0} 	\frac{ \varphi (c)}{c^{4s-2n}}\frac{(-1)^n \pi^2 \Gamma(4s-2n-1)^2}{4^{4s-2n-1}  \Gamma(2s-2n) \Gamma(2s-n)^2 \Gamma(2s)}(\Im (z_1)\Im (z_2))^{1+2n-4s},
$$
   where $\varphi(c)$ is the Euler's totient function.

    Observing that $\sum \limits_{c=1}\limits^{\infty}\dfrac{\varphi(c)}{c^s}=\dfrac{\zeta(s-1)}{\zeta(s)}$ is the Dirichlet series of $\varphi(c)$, we obtain
  \begin{equation}
\sum_{c>0}A^0(c)= \frac{(-1)^n \pi^2 \Gamma(4s-2n-1)^2}{4^{4s-2n-1}  \Gamma(2s-2n) \Gamma(2s-n)^2 \Gamma(2s)}\frac{\zeta(4s-2n-1)}{\zeta(4s-2n)}(\Im (z_1)\Im (z_2))^{1+2n-4s}. \label{A0}
  \end{equation}
The constant term $\sum \limits_{c>0}A^0(c)$ has no pole at the point $s = n = 1$, since the gamma function $\Gamma(z)$ has a simple pole at the origin.\\

\item Multiplying the constant term of the Fourier expansion of $S_{2n}\left(z_1+\frac{d_0}{c}, ~0, ~2s \right)$  by the second term of the expansion $S_0\left(z_2+\frac{a_0}{c}, ~0, ~2s-n \right)$, we get the second term in (\ref{Fourier}):
\begin{multline*}
 \sum_{c>0}A^{r} (c) = \frac{4^{1+n-2s} (-1)^n \pi ^{2s-n-1} \Gamma(4s-2n-1) }{\Gamma(2s-2n) \Gamma(2s-n)\Gamma(2s)}\Im (z_1)^{1+2n-4s} \Im(z_2)^{1/2+n-2s}\times \\
 \times \sum_{c>0} \left| r\right|^{2s-n-1/2} K_{2s-n-1/2}(2\pi\left|r\right|\Im(z_2))\sum_{\substack{1< a_0 < c\\(a_0, c)=1 }}\frac{1}{c^{4s-2n}} e^{2 \pi i ra_0/c}.
\end{multline*}

Note that $\sum \limits_{\substack{a_0=1, \\(a_0, c)=1}}\limits^{c}e^{2\pi i \frac{a_0}{c} r}=C_c(r)$ is the Ramanujan sum. The Dirichlet series for Ramanujan's sum is $\sum \limits_{c=1}\limits^{\infty}\dfrac{C_c(r)}{c^s}=\dfrac{\sigma_{1-s}(r)}{\zeta(s)}$, where $\sigma_{1-s}(r)=\sum \limits_{d|r}d^{1-s}$ is the divisor function.

Therefore,  	
$$ \sum_{c>0} \sum_{\substack{1< a_0 < c\\(a_0, c)=1 }}\frac{1}{c^{4s-2n}} e^{2 \pi i ra_0/c} = \frac{\sigma_{1+2n-4s}(r)}{\zeta(4s-2n)}.$$
Since the function $\sigma_{-2\lambda}(r)r^\lambda$ is an entire function in $\lambda$,
$K_{\lambda}(Y)$ is entire in $\lambda$ and  $K_\lambda(2Y)=O(e^{-\alpha Y})$ as $Y \longrightarrow \infty$, $0<\alpha<2$, it follows that the term $A^{r} (c)$ can be analytically continued to the point $s=n=1$.\\

\item Arguing as above, we can consider the term
\begin{multline*}
 \sum_{c>0}A^{r'} (c)= \frac{(-1)^n  \pi^{4s-2n-1/2} \Gamma(4s-2n-1)}{ 4^{2s-1} \Gamma(2s-n)^2 \Gamma(2s)} \Im (z_2)^{1+2n-4s}\times \\
 \times \sum_{c>0} \left| r'\right|^{4s-2n-1}\Phi_{\mathrm{sgn (r')}}\left(\left|\pi r' \Im (z_1)\right|,~ 2n,~ 2s-2n-\frac{1}{2}\right)\sum_{\substack{1< d_0 < c\\(d_0, c)=1 }}\frac{1}{c^{4s-2n}}e^{2 \pi i r' d_0/c},
 \end{multline*}	
which has no pole at the point $s=n=1$.\\

 \item Let $$C(n,s)=\frac{ (-1)^n \pi ^{6s-3n-1/2} }{ 4^{n-1} \Gamma(2s-n) \Gamma(2s)},$$
then the coefficient $A^{r, r'} (c)$ is given by the following:
\begin{multline*}
 \sum_{c>0}A^{r, r'} (c)=  C(n,s)\left| r\right|^{2s-n-1/2}\left|  r'\right|^{4s-2n-1} \Im(z_2)^{-2s+n+1/2}K_{2s-n-1/2}(2\pi\left|r\right|\Im(z_2))\times \\
 \times \Phi_{\mathrm{sgn (r')}}\left(\left|\pi r' \Im (z_1)\right|,~ 2n,~ 2s-2n-\frac{1}{2}\right)\sum_{c>0}\sum_{\substack{1< a_0 < c\\(a_0, c)=1 \\a_0d_0 \equiv 1 \pmod c}}\frac{1}{c^{4s-2n}}e^{2 \pi i \frac{ra_0+r'd_0}{c}}. \label{Arr'}
 \end{multline*}	 	
 	 	 	 	
Note that in the right-hand side of this expression we have a Kloosterman sum $$K(a,b; n)=\sum_{\substack{1\leq m< n, \\(m,n)=1,\\ mm^\ast\equiv 1 \pmod n}}e^{2 \pi i \left(\frac{am}{n}+\frac{bm^\ast}{n}\right)}.$$ 	

There is a well-known estimate \cite{2},~ \cite{8} for Kloosterman sums:

\begin{lemma} [Andre Weil]  Let $d(n)$ be the number of the positive divisors of $n$, then
$$\left|K(a,b; n)\right|\leq\left|n\right|^{1/2} min\left\{\sqrt{(a,n)}\cdot d\left(\frac{n}{(a,n)}\right),\sqrt{(b,n)}\cdot d\left(\frac{n}{(b,n)}\right)\right\}.$$ 	
 \end{lemma}

\begin{corollary} 	 	
If $a\geq 1$ is fixed, then for $\Re (s)>3/4$
\begin{equation}
\sum_{n\neq 0}\frac{\left|K(a,b; n)\right|}{\left|n\right|^{2s}}\leq \sqrt{a} \sum_{n \neq 0} \frac{d(n)}{\left|n\right|^{2s-1/2}}. \label{Weil}
\end{equation}
\end{corollary}

We know the Dirichel series involving the divisor function $d(c)=\sigma_0(n)$:
\begin{equation}
\sum_{c=1}^{\infty}\frac{d(c)}{c^s}=\zeta(s)^2. \label{d(n)}
\end{equation}

Finally, using (\ref{Weil}), (\ref{d(n)}), we obtain the estimate
\begin{multline*}
\sum_{c>0}A^{r, r'} (c)\leq C(n,s)\zeta(4s-2n-1/2)^2\left| r\right|^{2s-n-1/2}\left|  r'\right|^{4s-2n-1} \Im(z_2)^{1/2+n-2s}\times \\
\times K_{2s-n-1/2}(2\pi\left|r\right|\Im(z_2))\Phi_{sgn (r')}\left(\left|\pi r' \Im (z_1)\right|,~ 2n,~ 2s-2n-\frac{1}{2}\right).
\end{multline*}  	 	 	  	 	

Thus, this Fourier coefficient does not have a pole at the point $s=n=1$. The Bessel function $K_{\lambda}(2Y)$ is exponentially small in $Y$ as $Y \rightarrow \infty$, therefore the terms $\sum\limits_{c>0}A^{r} (c)$, $\sum\limits_{c>0}A^{r'} (c)$, $\sum\limits_{c>0}A^{r, r'} (c)$ are absolutely convergent.  This completes the proof of part (1) of the theorem. \\

 \end{enumerate}
\end{enumerate}

  \end{proof}

\textit{Let us define ~$\Xi_1(z_1,z_2)$ as the value of the holomorphic function $\Xi_1(z_1,z_2, s)$ at $s=1$}:
 $$\Xi_1(z_1,z_2) = \lim_{s\rightarrow 1}~ \Xi_1(z_1,z_2, s). $$
 From the behavior of ~$\Xi_1(z_1,z_2, s)$ under modular transformations we immediately obtain the behavior of the function $\Xi_1(z_1,z_2)$:

$$\Xi_1(\gamma z_1,z_2)= (cz_1+d)^2 ~\Xi_1(z_1,z_2), ~~~\textrm{for}~~~  \gamma=\left(\begin{array}{ccc}
	a&b\\
	c&d\\
	\end{array}\right) \in \Gamma.$$

\begin{remark} In order to define the sum $\omega(z_1,\bar{z_2}, k)= \sum\limits_{\gamma \in \Gamma} \dfrac{1}{\mu_{\gamma}(z_1, -\bar{z_2})^k}$ for $k=2$, let us consider the series
 	\begin{equation}
\Omega_{n}(z_1,\bar{z_2},s)=\sum \limits_{\gamma \in \Gamma}\frac{\mu_{\gamma}(\bar{z_1},-\bar{z_2})^{n-1}\mu_{\gamma}(\bar{z_1},-z_2)^{n+1}}{|\mu_{\gamma}(z_1, -z_2)|^{2s-2}|\mu_{\gamma}(z_1, -\bar{z_2})|^{2s+2}}.
		\end{equation}
Using the same argument as in the proof of (\ref{T3}), one can prove that
\begin{lemma}
The series $\Omega_{n}(z_1,\bar{z_2},s)$ can be analytically continued to the point~$s=n=1$.
\end{lemma}
Therefore, we can put $\omega_2(z_1,\bar{z_2})=\lim\limits_{s\rightarrow 1}\Omega_{1}(z_1,\bar{z_2},s)$.
\end{remark}

\section{Derivatives of the function $\Xi_1(z_1,z_2, s)$. }

In the previous section we have defined the so-called almost everywhere holomorphic modular form  $\Xi_1(\gamma z_1,z_2)$.
In this section we will calculate the derivatives of the function $(z_2-\bar{z_2})^{2s-1}\Xi_n(z_1,z_2, s)$ for $n=1$ and $\Re(s)>1$ and, by adding some term to $(z_2-\bar{z_2})^{2s-1}\Xi_1(z_1,z_2)$, we will obtain a quasi-modular form in $z_1$ of weight 2.

\begin{lemma}
$$\overline{\partial } ~\Xi_1(z_1,z_2)(z_2-\bar{z_2}) = \frac{-12 ~d\bar{z_1}}{(z_1-\bar{z_1})^2},  ~~~\textrm{where}~~~ \overline{\partial} = \frac{\partial}{\partial \bar{z_1}} d\bar{z_1} + \frac{\partial}{\partial \bar{z_2}} d\bar{z_2}.$$
\end{lemma}
\begin{proof}
\begin{itemize}
  \item Differentiation of the function $(z_2-\bar{z_2})^{2s-1}\Xi_1(z_1,z_2, s)$ with respect to $\bar{z_1}$ gives
  \begin{multline}
\frac{d}{d~\bar{z_1}}(z_2-\bar{z_2})^{2s-1}\Xi_1(z_1,z_2,s)=(s-1)\frac{(z_2-\bar{z_2})^{2s-1}}{z_1-\bar{z_1}} \times \left[ 2~\Xi_1(z_1,z_2,s)-\right.\\
- \left. \sum_{ad-bc=1}\frac{1}{|\mu_{\gamma}(z_1, -z_2)|^{2s}|\mu_{\gamma}(z_1, -\bar{z_2})|^{2s-2}}- \sum_{ad-bc=1}\frac{1}{|\mu_{\gamma}(z_1, -z_2)|^{2s-2}|\mu_{\gamma}(z_1, -\bar{z_2})|^{2s}}\right]
\end{multline}

Denote by $\Psi^1(z_1,z_2,s)$, $\Psi^2(z_1,z_2,s)$ the following sums:
$$\Psi^1(z_1,z_2,s)= \sum_{ad-bc=1} \frac{1}{|\mu_{\gamma}(z_1, -z_2)|^{2s}|\mu_{\gamma}(z_1, -\bar{z_2})|^{2s-2}},$$
$$ \Psi^2(z_1,z_2,s)= \sum_{ad-bc=1}\frac{1}{|\mu_{\gamma}(z_1,- z_2)|^{2s-2}|\mu_{\gamma}(z_1, -\bar{z_2})|^{2s}}.$$

One can show that the functions $\Psi^1(z_1,z_2,s)$, $\Psi^2(z_1,z_2,s)$ are holomorphic for $\Re(s)>1$ with the simple poles at $s=1$, of residues $\dfrac{3}{2}\dfrac{1}{\Im(z_1)\Im(z_2)}$.
Therefore, $$\lim\limits_{s \rightarrow 1}\frac{d}{d~\bar{z_1}}(z_2-\bar{z_2})^{2s-1}\Xi_1(z_1,z_2,s) = -\frac{12}{(z_1-\bar{z_1})^2}.$$
  \item The partial derivative with respect to $\bar{z_2}$ gives
  \begin{multline}
\frac{d}{d~\bar{z_2}}(z_2-\bar{z_2})^{2s-1}\Xi_n(z_1,z_2,s)=\\
=(z_2-\bar{z_2})^{2s-2}\left[(1-s)\sum_{ad-bc=1}\frac{\mu_{\gamma}(\bar{z_1},-z_2)^{2}}{|\mu_{\gamma}(z_1, -z_2)|^{2s}|\mu_{\gamma}(z_1, -\bar{z_2})|^{2s}}- s~\Omega_1(z_1, \bar{z_2},s) \right].
\end{multline}
It follows that
\begin{equation}
\lim_{s \rightarrow 1} \frac{d}{d~\bar{z_2}}\Xi_1(z_1,z_2,s)(z_2-\bar{z_2})^{2s-1}=-\sum_{ad-bc=1}\frac{1}{\mu_{\gamma}(z_1, \bar{z_2})^2}=-\omega_2(z_1, \bar{z_2}). \label{Dz_2Xi}
\end{equation}

\end{itemize}

 By the same argument, it can be shown that the function $\omega_2(z_1,\bar{z_2})$ is holomorphic with respect to $z_1$ and antiholomorphic with respect to $z_2$. Consequently, it is a cusp form of weight 2 in $z_1$  and thus vanishes. Hence, the function $\Xi_1(z_1,z_2)(z_2-\bar{z_2})$ is holomorphic in $z_2$.
   \end{proof}
    Thus by adding the term $-\dfrac{12}{z_1-\bar{z_1}}$ to the series $\Xi_1(z_1,z_2)(z_2-\bar{z_2})$, we get the function
\begin{equation}
\Xi^{\ast}_1(z_1,z_2)=  \Xi_1(z_1,z_2)(z_2-\bar{z_2}) - \frac{12}{z_1-\bar{z_1}},
\end{equation}
which is holomorphic in $z_1$, $z_2$.

\section{The differential form $\Xi^{\ast}(z_1,z_2)~dz_1$}

The function $\Xi_1(z_1,z_2)(z_2-\bar{z_2})$ is transformed under the action of $\Gamma$ like a modular form  of weight 2 with respect to $z_1$, hence
\begin{multline}
\Xi^{\ast}_1(\gamma z_1,z_2) = \Xi_1(\gamma z_1,z_2)(z_2-\bar{z_2})-12~\frac{|cz_1+d|^2}{z_1-\bar{z_1}}=\\
=(cz_1+d)^2~ \Xi^{\ast}_1(z_1,z_2) +12~c(cz_1+d) ~~~\textrm{for}~~~ \gamma=\left(\begin{array}{ccc}
	a&b\\
	c&d\\
	\end{array}\right).
\end{multline}
 Define the  following 1-form on $\Gamma \backslash \mathbb{H}$:
 \begin{equation}
 \Xi^{\ast}_1(z_1,z_2)(z_2-\bar{z_2})~dz_1 = \Xi_1(z_1,z_2)(z_2-\bar{z_2})~dz_1 - \frac{12~dz_1}{z_1-\bar{z_1}}.
 \end{equation}

 Let $\Delta(z)$ denote the modular discriminant: $$\Delta(z) = \frac{1}{1728}(E_4(z)^3-E_6(z)^2), ~~~\textrm{where}~~~E_k(z)=\frac{1}{2}\sum_{\substack{c,d \in \mathbb{Z} \\ (c,d)=1}}\frac{1}{(cz+d)^k} ~~~\textrm{is  the Eisenstein series.}~~~$$ The modular discriminant $\Delta(z)$  is a cusp form of weight 12. It follows in the standard way that $$ \frac{d}{dz_1}\log~\Delta(\gamma z)=\frac{\Delta'(\gamma z)}{\Delta(\gamma z)}=\frac{d}{dz_1}\log~\Delta(z)+12~c(cz+d).$$
 The modular invariant $j(z)$ is defined by $$j(z)=\frac{E_4(z)^3}{\Delta(z)}.$$
   \begin{theorem} One has
\begin{equation}\Xi^{\ast}(z_1,z_2,1)(z_2-\bar{z_2})~dz_1 = d_{z_1} \log~ \left| (j(z_1)-j(z_2))~ \Delta(z_1)\Delta(z_2) \right|^2. \label{T2}
\end{equation}
 \end{theorem} 	

 \begin{proof}

  Since the $j$-invariant is a modular function of weight 0 with a simple pole at the cusp and holomorphic on $\mathbb{H}$, then the meromorphic 1-form $d_{z_1} \log~|j(z_1)-j(z_2)|^2$ has the simple poles at the points $z_1=\gamma z_2$ and at the cusps.
Thus the differential forms in the left-hand side and in the right-hand side of (\ref{T2}) have the same poles with the same residues. As a consequence, the difference of this differential forms is a holomorphic 1-form invariant under the modular transformation. Therefore, it is equal to zero.

 \end{proof}

\section{Acknowledgements}
I am very grateful to A. Levin for his great ideas, help and an inspiration.

%%%%%%%%%%%%%%%%%%%%%%%%%%%%%%%%%%%%%%%%%%%%%%%%%%%%%%%%%%%%%%%%%%%%%%%%%%

\bibliographystyle{amsplain}

\end{document}